\documentclass[a4paper]{amsart}
\usepackage[utf8]{inputenc}
\usepackage[T1]{fontenc}

\usepackage{geometry}
\usepackage{amssymb}
\usepackage{amsfonts}
\usepackage{amsthm}
\usepackage{calrsfs}
\usepackage{array}
\usepackage{bbm}
\usepackage{stmaryrd}
\usepackage{hyperref}
\usepackage{enumerate}
\usepackage{mathabx}
\usepackage{enumitem}
\usepackage[demo]{graphicx}
\usepackage{caption}
\usepackage{subcaption}
\usepackage{stmaryrd}

\usepackage{csquotes}

\usepackage{mathtools}
\mathtoolsset{showonlyrefs=true}
\usepackage[svgnames]{xcolor}
\usepackage{hyperref}

\usepackage{tikz}
\usetikzlibrary{positioning}

\newcommand{\R}{\mathbb{R}}

\newcommand{\jap}[1]{\langle #1 \rangle}

\newcommand{\fia}{\mathbbm{1}_{|\Phi-\alpha|<M}}

\theoremstyle{plain}
\newtheorem{thm}{Theorem}[section]
\newtheorem*{thm*}{Theorem}
\newtheorem{prop}[thm]{Proposition}

\newtheorem{lem}[thm]{Lemma}

\theoremstyle{definition}

\theoremstyle{remark}
\newtheorem{nb}[thm]{Remark}

\numberwithin{equation}{section}
\newtheoremstyle{mytheoremstyle} 
{\topsep}                    
{\topsep}                    
{}                   
{}                           
{\scshape}                   
{.}                          
{.5em}                       
{}  

\theoremstyle{mytheoremstyle} 
\theoremstyle{mytheoremstyle} 

\makeatletter
\@namedef{subjclassname@2020}{%
	\textup{2020} Mathematics Subject Classification}
\makeatother

\date{}
\author{Simão Correia and Pedro Leite}

\title{Sharp local existence and nonlinear smoothing for dispersive equations with higher-order nonlinearities}


\thanks{S. C. was partially supported by Funda\c{c}\~ao para a Ci\^encia e Tecnologia, through CAMGSD, IST-ID
	(projects UIDB/04459/2020 and UIDP/04459/2020). S. C. and P. L. were  partially supported by Funda\c{c}\~ao para a Ci\^encia e Tecnologia, through the project NoDES (PTDC/MAT-PUR/1788/2020).}

\allowdisplaybreaks

\begin{document}
\maketitle
\begin{abstract}
	We consider a general nonlinear dispersive equation with monomial nonlinearity of order $k$ over $\R^d$. We construct a rigorous theory which states that higher-order nonlinearities and higher dimensions induce sharper local well-posedness theories. More precisely, assuming that a certain positive multiplier estimate holds at order $k_0$ and in dimension $d_0$, we prove a sharp local well-posedness result in $H^s(\R^d)$ for any $k\ge k_0$ and $d\ge d_0$. Moreover, we give an explicit bound on the gain of regularity observed in the difference between the linear and nonlinear solutions, confirming the conjecture made in \cite{cos}. The result is then applied to generalized Korteweg-de Vries, Zakharov-Kuznetsov and nonlinear Schrödinger equations.
\end{abstract}
\section{Introduction}
\subsection{Setting and motivation of the problem} We consider a nonlinear dispersive equation of the form
\begin{equation}\label{eq:geral}
\partial_tv + iL(D)v = N[v],\qquad v(0)=v_0\in H^s(\R^d),
\end{equation}
where $L(D)$, $D=\nabla/i$, is a spatial pseudo-differential operator defined, in frequency variables, by a real smooth multiplier $L(\xi)$, and $N$ is a nonlinear term involving $k$ powers of either $u$ or $\bar{u}$. Examples include the generalized Korteweg-de Vries equation,
\begin{equation}\label{gkdv}\tag{gKdV}
	\partial_tv + \partial_x^3v = \partial_x(v^k),
\end{equation}
the generalized Zakharov-Kuznetsov equation over $\R^d$,
\begin{equation}\label{gzk}\tag{gZK}
	\partial_tv + \partial_x\Delta v = \partial_x(v^k),
\end{equation}
and the nonlinear Schrödinger equation over $\R^d$,
\begin{equation}\label{nls}\tag{NLS}
	i\partial_tv +\Delta v = \pm|v|^{k-1}v.
\end{equation}

The problem of determining the precise range of regularities for which the initial-value problem \eqref{eq:geral} is locally well-posed in $H^s(\R^d)$ has been the focus of intensive research for the past sixty years. It is impossible to give a detailed list of references, as each individual equation has its own features and consequently its own local well-posedness theory; we simply refer to the monographs \cite{caz_book, tzirakis_book, linaresponce_book, tao_book}.

To simplify our discussion, let us assume that $L(\xi)$ is homogeneous of order $\ell$ (see Remark \ref{rem:L} below for the general case) and that the nonlinear term presents a derivative loss of order $n$. In this case, the equation admits a scaling invariance: if $v$ is a solution, so is
$$
v_\lambda(t,x)=\lambda^{\frac{\ell-n}{k-1}}v(\lambda^\ell t, \lambda x).
$$
As such, the critical regularity exponent (that is, the value of $s$ for which the scaling leaves the $\dot{H}^s$ norm invariant) is given by
\begin{equation}
	s_c=s_c(k,d)=\frac{d}{2}-\frac{\ell-n}{k-1}.
\end{equation}
For regularities $s<s_c$, one generally expects ill-posedness, as larger initial data paradoxically yield longer times of existence (see, for example, the discussion in \cite[Section 3.1]{tao_book}). The problem of determining the optimal threshold $s_{\text{LWP}}\ge s_c$ for which local well-posedness holds for $s>s_{\text{LWP}}$ is an extremely hard problem and depends on various factors: the dispersion $\ell$, the derivative loss $n$, the dimension $d$, the order of the nonlinearity $k$ and even on the specific structure of the dispersion and the nonlinear term. 

Despite the complexity of the problem at hand, one can make two heuristic arguments. First, higher values of $k$ induce weaker short-time nonlinear effects and thus should allow for larger ranges of local well-posedness. Second, higher dimensions yield stronger local-in-time dispersion, opening the door for an improved regularity threshold. As such, one may expect the following statement:
\begin{equation}\label{eq:folklore}
	\textit{If $s_c=s_{\text{LWP}}$ for some $k=k_0$ and $d=d_0$, then it should also hold for $k\ge k_0$ and $d\ge d_0$.}
\end{equation}

In practice, the validity of the above statement has always been checked \textit{a posteriori}, after proving sharp local well-posedness results for all $k\ge k_0$ (independently). As such, while a rigorous proof of this conjecture is currently unavailable, it has been generally accepted as ``folklore'' among the dispersive community.

Another related problem concerning \eqref{eq:geral} is the nonlinear smoothing effect, which asserts that the difference between the linear and nonlinear evolutions with the same initial data  is actually smoother than the initial data. This can be seen as a form of semilinearity of the equation in terms of regularity, as the equation is a ``regular'' perturbation of the linear equation. This feature was first observed in \cite{bonasaut} and, since then, several works prove nonlinear smoothing for various dispersive equations \cite{bpss,  tzirakis1, tzirakis2,   isazamejiatzevtkov, imos, keraanivargas, linarespastorsilva, linaresponcesmith,   linaresscialom} 

In the more recent references cited above, nonlinear smoothing has been studied in terms of a gain in Sobolev regularity. That is, given initial data in $H^s$, prove that the difference between the linear and nonlinear evolutions is in $H^{s+\epsilon}$ for all $t$. In this direction, the first author, together with Oliveira and Silva, conjectured in \cite{cos} that, for $k\ge 3$, one should be able to prove nonlinear smoothing for
\begin{equation}\label{eq:smoothing}
	\epsilon<\epsilon_c(k,d,s):=\min\{ (k-1)(s-s_{\text{LWP}}), \ell-n-1 \}.
\end{equation} 
To put it briefly, the bound $\ell-n-1$ is connected to a $high\times low \times \dots \times low\to high$ interaction, where the dispersion can be fully exploited but where the order of the nonlinearity is irrelevant. On the other hand, the bound $(k-1)(s-s_{\text{LWP}})$ is related to a $high\times high \times \dots \times high\to high$ interaction, where one may transfer freely the extra regularity $s-s_{\text{LWP}}$ from the ingoing frequencies to the outgoing frequency.

\medskip
The main goal of this work is to provide a concrete theoretical basis for statement \eqref{eq:folklore} and conjecture \eqref{eq:smoothing}. Indeed, we will prove that, whenever a specific positive multiplier estimate - called \textit{flexible frequency-restricted estimate} -  holds for some $k=k_0$ and $d=d_0$, then, \textit{for all} $k\ge k_0$ \textit{and} $d\ge d_0$, $s_c=s_{\text{LWP}}$ and \eqref{eq:smoothing} holds. 

\subsection{Applications}

Before we move to the technical description of the abstract induction results that validate \eqref{eq:folklore} and \eqref{eq:smoothing}, let us first present three different applications.

\medskip

First, consider the generalized KdV equation with $k\ge 5$,
\begin{equation}
	\partial_tv + \partial_x^3v = \partial_x(v^k).
\end{equation}
The initial-value problem in locally well-posed in $H^s(\R)$ for any $s>s_c(k,d)=\frac{1}{2}-\frac{2}{k-1}$ (\cite{kpv_gkdv}). The proof of this result uses a combination of local smoothing estimates and maximal estimates, together with a Leibniz rule and a chain rule for fractional derivatives. Concerning the nonlinear smoothing property, there are currently no results in this direction. In Section \ref{sec:kdv}, we will prove the validity of the flexible frequency-restricted estimate for $k_0=5$. In particular,
\begin{thm}\label{thm:gkdv}
	Given $k\ge 5$, equation \eqref{gkdv} is locally well-posed in $H^s(\R)$, for any $s>s_c(k,1)=\frac{1}{2}-\frac{2}{k-1}$. Moreover, the corresponding flow presents a nonlinear smoothing effect of order
	$$
	\epsilon<\epsilon_c(k,1,s)=\min\left\{(k-1)\left(s-\frac{1}{2}+\frac{2}{k-1}\right), 1\right\}.
	$$
\end{thm}

\begin{nb}
	For the generalized KdV, one actually has $s_{\text{LWP}}=s_c$ for $k\ge 4$. Since the case $k=4$ has already been studied in \cite{cos} and the proof of the flexible frequency-restricted estimate is more challenging when $k=4$, we have decided to restrict ourselves to $k\ge 5$.
\end{nb}

\medskip

The second application concerns the generalized Zakharov-Kuznetsov equation
\begin{equation}
	\partial_tv + \partial_x\Delta v = \partial_x(v^k).
\end{equation}
Let us first consider the case $d=2$ and $k\ge 5$. In \cite{ribaudvento}, Ribaud and Vento proved the local well-posedness of \eqref{gzk} for $s>s_c(k,2)=1-\frac{2}{k-1}$, using ideas similar to those of the \eqref{gkdv} case. Regarding nonlinear smoothing, the only results available concern $k=2,3$ (\cite{cos, linarespastorsilva}) and nothing is known for $k\ge 5$. As observed in \cite{herrgrunrock}, the \eqref{gzk} equation is equivalent to
\begin{equation}\label{zksym}\tag{gZK\textsubscript{sym}}
	\partial_tv + (\partial^3_{x}+\partial_{y}^3) v = (\partial_x + \partial_y)(v^k).
\end{equation}
Observe that the linear part is now a sum of two independent one-dimensional operators, each of them related to \eqref{gkdv}
As such, we are able to transfer the sharp results for the \eqref{gkdv} to \eqref{gzk} when $d=2$:
 \begin{thm}\label{thm:2dzk}
 	Fix $d=2$. Given $k\ge 5$, equation \eqref{gzk} is locally well-posed in $H^s(\R^2)$, for any $s>s_c(k,2)=1-\frac{2}{k-1}$. Moreover, the corresponding flow presents a nonlinear smoothing effect of order
 	$$
 	\epsilon<\epsilon_c(k,2,s)=\min\left\{(k-1)\left(s-1+\frac{2}{k-1}\right), 1\right\}.
 	$$
 \end{thm}
For dimensions $d\ge 3$ and $k\ge 5$, the local existence was shown for $s>s_c(k,d)$ by Linares and Ramos \cite{linaresramos_gzk}. For $d=3$ and $k\ge 3$, the optimal result was proved by Grünrock \cite{grunrock_cubiczk, grunrock_gzk}. For $d\ge 4$ and $k=3$, Kinoshita proved small data global well-posedness in $H^{s_c}(\R^d)$ \cite{kinoshita_3zk}. This was complemented by \cite{cos}, where the local well-posedness was obtained for $s>s_c(3,d)$ and a nonlinear smoothing effect of order $\epsilon<\epsilon_c(3,d)$ was shown to hold.

We adapt the proof of \cite{cos} to prove the flexible frequency-restricted estimate for \eqref{gzk} for $k_0=3$ and $d_0\ge 3$. As such, we are able to induce the result of \cite{cos} to larger values of $k$:

 \begin{thm}\label{thm:gzk}
	Fix $d\ge3$. Given $k\ge 3$, equation \eqref{gzk} is locally well-posed in $H^s(\R^d)$, for any $s>s_c(k,d)=\frac{d}{2}-\frac{2}{k-1}$. Moreover, the corresponding flow presents a nonlinear smoothing effect of order
	$$
	\epsilon<\epsilon_c(k,d,s)=\min\left\{(k-1)\left(s-\frac{d}{2}+\frac{2}{k-1}\right), 1\right\}.
	$$
\end{thm}

\medskip

Our last application concerns the nonlinear Schrödinger equation in $\R^d$, $d\ge 2$,
\begin{equation}
	i\partial_tv +\Delta v = \pm|v|^{k-1}v,\qquad k\mbox{ odd}.
\end{equation}
The local well-posedness for $s>s_c(k,d)$ is standard (see, for example, \cite{cazweiss_criticalHs,pecher_nls,kato_nls}). Concerning the nonlinear smoothing property, we refer to \cite{bpss}, where the authors prove, for $s>\frac{d}{2}-\frac{1}{2(k-1)}$, a nonlinear smoothing effect of order $\epsilon=1^-$. This result was improved in \cite{cos} for $k=3, 5$, where nonlinear smoothing of order $\epsilon<\epsilon_c(k,d,s)$ was shown for all $s>s_c(k,d)$.

Here we will prove the flexible frequency-restricted estimate for \eqref{nls} when $k_0=3$ and $d_0\ge 2$. As a consequence, we derive

 \begin{thm}\label{thm:nls}
	Fix $d\ge2$. Given $k\ge 3$ odd, equation \eqref{nls} is locally well-posed in $H^s(\R^d)$, for any $s>s_c(k,d)=\frac{d}{2}-\frac{2}{k-1}$. Moreover, the corresponding flow presents a nonlinear smoothing effect of order
	$$
	\epsilon<\epsilon_c(k,d,s)=\min\left\{(k-1)\left(s-\frac{d}{2}+\frac{2}{k-1}\right), 1\right\}.
	$$
\end{thm}

Comparing with the existing literature, we remark that there is a substantial gain when $k\ge 7$, regardless of the dimension.

\begin{nb}
	In the proof of Theorem \ref{thm:nls}, the exact number of conjugates in the nonlinear term is irrelevant. Furthermore, the condition that $k$ is odd is only used to ensure that $|v|^{k-1}v$ is indeed a polynomial in $v$ and $\bar{v}$. As such, the theorem holds verbatim for the equation
	$$
		i\partial_tv +\Delta v = v^{k'}\bar{v}^{k-k'},\quad 0\le k'\le k,
	$$
for any $k\ge 3$.
\end{nb}

\begin{nb}
	As a consequence of our induction results, one may consider the local well-posedness problem in $H^s$, $s\ge d/2$, for analytic nonlinearities of the form
	$$
	N[v]=\sum_{k\ge k_0}c_k N_k[v],\quad N_k\mbox{ of order }k,
	$$
	where $c_k$ decays faster than any exponential sequence. Indeed, as one can check in the proof of Theorem \ref{thm:ind_k} below, the bounds on the individual nonlinear terms grow as an exponential factor in $k$. In particular, one may easily obtain an \textit{a priori} bound on $N[v]$ for $s\ge d/2>s_c(k,d)$ (fo all $k$), and local well-posedness follows. Even further, one can derive a nonlinear smoothing result for any $s\ge d/2$ and $\epsilon<\min_{k\ge k_0} \epsilon_c(k,d,s)=\ell-n-1$.
\end{nb}

\begin{nb}\label{rem:L}
One may also consider the case where $L$ includes some lower-order terms. The induction theorems in either $k$ or $d$ still hold, as the proofs do not require $L$ to be homogeneous. The only caveat is that one should be able to ensure that, in the flexible frequency-restricted estimate, it is the higher-order terms which dictate the regularity threshold $s_c$.
\end{nb}

The above examples are a clear indication of the applicability of our abstract results. In our perspective, the induction results are an important step towards a unification of the local well-posedness theories for general semilinear dispersive equations.

\subsection{Statement of the abstract induction results}
Now that we covered the applications of our abstract results, we move to their precise formulation. As it is well-known (see, for example, \cite[Section 7.4]{linaresponce_book}), if one considers the Bourgain space $X^{s,b}$ associated with \eqref{eq:geral} (\cite{bourg1, bourg2}), defined by the norm\footnote{$\mathcal{F}_{t,x}$ denotes the space-time Fourier transform, while $\hat{\cdot}$ represents the spatial Fourier transform.}
$$
\|v\|_{X^{s,b}} = \|\jap{\tau-L(\xi)}^b\jap{\xi}^s\mathcal{F}_{t,x}v\|_{L^2_{\tau,\xi}},
$$
then the initial-value problem \eqref{eq:geral} is locally well-posed in $H^s(\R^d)$ if, for some $b>1/2$ and $b'>b-1$,
\begin{equation}\label{eq:N1}
	\|N[u]\|_{X^{s,b'}}\lesssim \|u\|_{X^{s,b}}^k
\end{equation}
and
\begin{equation}\label{eq:N2}
	\|N[u]-N[v]\|_{X^{s,b'}}\lesssim (\|u\|_{X^{s,b}}^{k-1} + \|u\|_{X^{s,b}}^{k-1})\|u-v\|_{X^{s,b}}.
\end{equation}
Moreover, if, for some $\epsilon>0$, 
\begin{equation}\label{eq:N3}
	\|N[u]\|_{X^{s+\epsilon,b'}}\lesssim \|u\|_{X^{s,b}}^k,
\end{equation} 
holds, then the flow generated by \eqref{eq:geral} presents a nonlinear smoothing effect of order $\epsilon$. 

Assume that the nonlinearity $N$ is of multilinear form, that is, that there exists a $k$-linear operator (which we also call $N$ for abuse of notation) such that, in spatial frequency space,
$$
(N[u_1,\dots,u_k])^\wedge(\xi) = \int_{\xi=\xi_1+\dots+\xi_k} m_k(\xi_1,\dots,\xi_k)\left(\prod_{j=1}^{k'} \hat{u}(\xi_j)\right)\left(\prod_{j=k'+1}^k \hat{\bar{u}}(\xi_j)\right)d\xi_1\dots d\xi_{k-1}
$$
and
$$
N[u,\dots, u] = N[u].
$$
Then \eqref{eq:N1}-\eqref{eq:N3} are an immediate consequence of the multilinear estimate
\begin{equation}\label{eq:multi}
	\|N[u_1,\dots, u_k]\|_{X^{s+\epsilon,b'}}\lesssim \prod_{j=1}^k \|u_j\|_{X^{s,b}}.
\end{equation}

In \cite{cos}, the derivation of \eqref{eq:multi} has been reduced even further. Define the convolution hyperplane with deformation $\sigma$,
$$
\Gamma_\xi^\sigma = \left\{ (\xi_1,\dots,\xi_k)\in (\R^{d})^k: \xi+\sigma=\sum_{j=1}^{k'}\xi_j -\sum_{j=k'+1}^k \xi_j , \quad \min_j|\xi_j|\gtrsim |\sigma| \right\},
$$
and the resonance function 
$$
\Phi=L(\xi)-\sum_{j=1}^{k'}L(\xi_j) + \sum_{j=k'+1}^kL(\xi_j).
$$
For convenience of notation, we set $\xi_0=\xi$, $\Gamma_\xi=\Gamma_\xi^0$, $\llbracket a,b \rrbracket=[a,b]\cap \mathbb{Z}$ and write ``$\xi_j, j\in A$'' as ``$\xi_{j\in A}$''.
 \begin{lem}\cite[Lemma 2]{cos}\label{lem:interpol}
	Suppose that there exist $\emptyset\neq A \subsetneq\llbracket0,k\rrbracket$ and $\mathcal{M}_j=\mathcal{M}_j(\xi,\xi_1,\dots,\xi_k)\ge 0$, $j=1,2$, such that
	$$
	\left(\mathcal{M}_1\mathcal{M}_2\right)^{\frac{1}{2}}=\frac{|m_k(\xi_1,\dots, \xi_k)|\jap{\xi}^{s+\epsilon}}{\prod_{j=1}^k\jap{\xi_j}^{s}}
	$$
	and, for any $M>1$, one has the frequency-restricted estimate
	\begin{equation}\label{eq:fre}
		\sup_{\xi_{j\in A},\alpha} \int_{\Gamma_\xi}\mathcal{M}_1\fia d\xi_{j\notin A} + 	\sup_{\xi_{j\notin A},\alpha} \int_{\Gamma_\xi}\mathcal{M}_2\fia d\xi_{j\in A}  \lesssim M^{1^-}.
	\end{equation}
	Then \eqref{eq:multi} holds for some $b>1/2$ and $b'>b-1$.
\end{lem}

The above lemma reduces the problem of induction of local well-posedness and nonlinear smoothing in $k$ and $d$ to an induction on frequency-restricted estimates. As it turns out, it is not true that a frequency-restricted estimate for a nonlinearity of order $k_0$ implies a frequency-restricted estimate at order $k\ge k_0$. To be able to induce the estimates, we need to add some flexibility.

\medskip
Given $s>s_c(k,d)$ and $\epsilon<\epsilon_c(k,d,s)$, we say that a \textit{flexible frequency-restricted estimate} holds if there exist  $\emptyset\neq A \subsetneq \llbracket0,k\rrbracket$, pairs $(s_j,s_j')$, $j\in\llbracket0,k\rrbracket$ and $(\epsilon_0,\epsilon_0')$ such that
$$
\begin{cases}
	s_j+s_j'=2s, \quad j\in\llbracket0,k\rrbracket,\\
	\epsilon_0+\epsilon_0'=2\epsilon,
\end{cases}
$$
and, for any $M\ge 1$,
\begin{equation}\label{eq:ffre1}
	\sup_{\sigma, \alpha, \xi_{j \in A}} \int_{\Gamma_{\xi}^{ \sigma}}\frac{|m_k(\xi_1,\dots,\xi_k)| \langle \xi \rangle ^{s_0+\epsilon_0}}{\prod_{j=1}^{k} \langle \xi_j \rangle ^{s_j}} \mathbbm{1}_{|\Phi-\alpha| < M} d\xi_{j \in A^c} \lesssim M^{1^-}
\end{equation}
and
\begin{equation}\label{eq:ffre2}
	\sup_{\sigma, \alpha, \xi_{j \in A^c}} \int_{\Gamma_{\xi}^{\sigma}}\frac{|m_k(\xi_1,\dots,\xi_k)| \langle \xi \rangle ^{{s'_0}+{\epsilon'_0}}}{\prod_{j=1}^{k} \langle \xi_j \rangle ^{{s'_j}}} \mathbbm{1}_{|\Phi-\alpha| < M} d\xi_{j \in A} \lesssim M^{1^-}.
\end{equation}
Before we state the theorem, we need to be able to relate equations at different orders $k$. First, given $k_0\le k$, we suppose that
\begin{equation}\label{eq:ind_multiplier}
	|m_k(\xi_1,\dots, \xi_k)| \lesssim |m_{k_0}(\tilde{\xi}_1,\dots,\tilde{\xi}_{k_0})|
\end{equation}
where $(\tilde{\xi}_1,\dots,\tilde{\xi}_{k_0})$ is the vector consisting in the $k_0$ largest (up to multiplicative constants) components of $(\xi_1,\dots, \xi_k)$.

Finally, given $k_0\le k$, we say that $\Phi_{k_0}:(\R^d)^{k_0+1}\to \R$ is a $k_0$-descent of $\Phi$ if there exists a set $J\subset\llbracket 1, k \rrbracket$, $|J|=k-k_0$, such that
$$
\Phi_{k_0}\cong\Phi\Big|_{H_J},\quad \mbox{where }H_J=\{(\xi, \xi_1,\dots,\xi_k)\in (\R^d)^{k+1}: \xi_j=0, \forall j\in J\}.
$$

We observe that $\Phi_{k_0}$ is a resonance function for a nonlinearity of order $k_0$ (which is derived from the nonlinearity of order $k$ by ``erasing'' $k-k_0$ factors).

\begin{nb}
	If no complex conjugates are present in $N$, the only $k_0$-descent of $\Phi$ is
$$
\Phi_{k_0}=L(\xi)-\sum_{j=1}^{k_0}L(\xi_j).
$$
However, under the presence of conjugates (as for the \eqref{nls}), it becomes important which $\xi_j$'s are being taken equal to 0, as that influences the signs on the $k_0$-descent of $\Phi$.
\end{nb}

\begin{thm}[Induction in $k$]\label{thm:ind_k}
Fix $k\ge 3$. Suppose that there exists $k_0< k$ such that, for all $s>s_c(k_0,d)$, $\epsilon<\epsilon_c(k_0,d)$ and for all $k_0$-descents of $\Phi$, a flexible frequency-restricted estimate holds. 	Then, given $s>s_c(k,d)$ and $\epsilon<\epsilon_c(k,d,s)$, the conditions of Lemma \ref{lem:interpol} are satisfied.

In particular, \eqref{eq:geral} is locally well-posed in $H^s(\R^d)$, $s>s_c(k,d)$, and the correspondng flow presents a nonlinear smoothing effect of order $\epsilon<\epsilon_c(k,d,s)$.
\end{thm}

Our last result concerns the induction to dimension $d$ of a given frequency-restricted estimate in dimension $d_0\le d$. To that end, we must impose that the linear operator $L$ can be decomposed as\footnote{Otherwise, it would become unclear how one relates an equation in $d$ dimensions with one in lower dimension.}
\begin{equation}\label{eq:Ld}
	L(\xi)=\sum_{j=1}^d L_0(\xi^j),\quad \xi=(\xi^1,\dots,\xi^d).
\end{equation}
For $k$ fixed, let $\Phi_0$ be the resonance function associated with the linear operator in $d_0$ dimensions. Moreover, we suppose that the multiplier in the nonlinear term depends only on $\xi$, $m=m(\xi)$, and that there exists a multiplier $m_0:\R^{d_0} \to \R$ such that $|m(\xi)|\lesssim |m_0(\eta)|$, where $\eta\in \R^{d_0}$ is the vector with the largest $d_0$ components of $\xi$.
\begin{thm}[Induction in $d$]\label{thm:ind_d} Suppose that there exist $\emptyset\neq A \subsetneq\llbracket0,k\rrbracket$ and positive multipliers $\mathcal{M}_j:(\R^{d_0})^{k+1}\to \R^+$, $j=1,2$ such that
	$$
	\left(\mathcal{M}_1\mathcal{M}_2\right)^{\frac{1}{2}}=\frac{|m_0(\xi )|\jap{\xi}^{s+\epsilon}}{\prod_{j=1}^k\jap{\xi_j}^{s}}
	$$
	and, for any $s>s_c(k,d_0)$, $\epsilon<\epsilon_c(k,d_0)$ and $M>1$,
	\begin{equation}\label{eq:fre_d}
		\sup_{\xi_{j\in A},\alpha} \int_{\Gamma_\xi}\mathcal{M}_1\mathbbm{1}_{|\Phi_0-\alpha|<M} d\xi_{j\notin A} + 	\sup_{\xi_{j\notin A},\alpha} \int_{\Gamma_\xi}\mathcal{M}_2\mathbbm{1}_{|\Phi_0-\alpha|<M} d\xi_{j\in A}  \lesssim M^{1^-}.
	\end{equation}
Furthermore, suppose that $s_c(k,d_0)\ge0$. Then, for any $d\ge d_0$, $s>s_c(k,d)$ and $\epsilon<\epsilon_c(k,d)$, \eqref{eq:multi} holds for some $b>1/2$ and $b'>b-1$.

In particular, \eqref{eq:geral} is locally well-posed in $H^s(\R^d)$, $s>s_c(k,d)$, and the corresponding flow presents a nonlinear smoothing effect of order $\epsilon<\epsilon_c(k,d,s)$. 
\end{thm}

\bigskip

This article is organized as follows. In Section \ref{sec:abstra}, we prove Theorems \ref{thm:ind_k} and \ref{thm:ind_d}. Once we have these results in hand, we apply them to the generalized KdV (Section \ref{sec:kdv}), generalized Zakharov-Kuznetsov (Section \ref{sec:zk}) and nonlinear Schrödinger equations (Section \ref{sec:nls}).

\section{Proof of the induction results}\label{sec:abstra}

\begin{proof}[Proof of Theorem \ref{thm:ind_k}]
To simplify the exposition, we assume that there are no complex conjugates present in the nonlinearity (for the general case, see Remark \ref{rem:conjugate} below). Without loss of generality, we order the frequencies
\begin{equation}\label{eq:order}
	|\xi_1|\gtrsim \dots \gtrsim |\xi_k|,
\end{equation}
which implies in particular that $|\xi_1|\gtrsim |\xi|$. 

Given $s>s_c(k,d)$ and $\epsilon<\epsilon_c(k,d,s)$, define\footnote{The notation $d^+$ represents a number slightly larger than $d$.}
\begin{equation}
	\tilde{s}=s-\Delta s,\quad \Delta s=\begin{cases}
\frac{(k-k_0)(d^+-2s)}{2(k_0-1)},& s\le d/2\\ 0 & s>d/2
	\end{cases}.
\end{equation}
Since $s>s_c(k,d)$, a simple computation shows that $\tilde{s}>s_c(k_0,d)$ and $\epsilon<\epsilon_c(k_0,d,\tilde{s})$. By assumption, there exist  $\emptyset\neq A \subsetneq \llbracket0,k\rrbracket$, pairs $(s_j,s_j')$, $j\in\llbracket0,k_0\rrbracket$ and $(\epsilon_0,\epsilon_0')$ such that
$$
\begin{cases}
	s_j+s_j'=2\tilde{s}, \quad j\in\llbracket0,k_0\rrbracket\\
	\epsilon_0+\epsilon_0'=2\epsilon
\end{cases}
$$
and
\begin{equation}
	\sup_{\sigma, \alpha, \xi_{j \in A}} \int_{\Gamma_{\xi}^{ \sigma}}\frac{|m_{k_0}| \langle \xi \rangle ^{s_0+\epsilon_0}}{\prod_{j=1}^{k_0} \langle \xi_j \rangle ^{s_j}} \mathbbm{1}_{|\Phi_{k_0}-\alpha| < M} d\xi_{j \in A^c} + 	\sup_{\sigma, \alpha, \xi_{j \in A^c}} \int_{\Gamma_{\xi}^{\sigma}}\frac{|m_{k_0}| \langle \xi \rangle ^{{s'_0}+{\epsilon'_0}}}{\prod_{j=1}^{k_0} \langle \xi_j \rangle ^{{s'_j}}} \mathbbm{1}_{|\Phi_{k_0}-\alpha| < M} d\xi_{j \in A} \lesssim M^{1^-}.
\end{equation}
We now prove that the conditions of Lemma \ref{lem:interpol} hold true. Set $B=A$, $B^c=A^c\cup \llbracket k_0+1, k\rrbracket$,
$$
\mathcal{M}_1=\frac{|m_k(\xi_1,\dots, \xi_k)| \langle \xi \rangle ^{s_0+\epsilon_0}}{\prod_{j=1}^{k_0} \langle \xi_j \rangle ^{s_j}}
$$
and
$$
\mathcal{M}_2=\frac{|m_k(\xi_1,\dots,\xi_k)| \langle \xi \rangle ^{s'_0+2\Delta s+\epsilon'_0}}{\prod_{j=1}^{k_0} \jap{\xi_j}^{s'_j+2\Delta s} \cdot \prod_{j=k_0+1}^k \jap{\xi_j}^{2s}}.
$$
The choice of weights ensures that
$$
(\mathcal{M}_1\mathcal{M}_2)^{\frac{1}{2}}=\frac{|m_k(\xi_1,\dots,\xi_k)| \langle \xi \rangle ^{s+\epsilon}}{\prod_{j=1}^{k} \langle \xi_j \rangle ^{s}}.
$$
Observe that, as the frequencies are ordered,
$$
\sigma:=-\sum_{j=k_0+1}^k \xi_j \quad \mbox{satisfies}\quad |\sigma|\lesssim \min_{1\le j\le k_0} \{|\xi_j|\}.
$$
In particular, if $(\xi_1,\dots,\xi_k)\in \Gamma_\xi$, then $(\xi_1,\dots, \xi_{k_0})\in \Gamma_\xi^\sigma$.
Therefore
\begin{align*}
\sup_{\alpha, \xi_{j \in B^c}} \int_{\Gamma_{\xi}}\mathcal{M}_1 \mathbbm{1}_{|\Phi-\alpha| < M} d\xi_{j \in B}=	&\sup_{\alpha, \xi_{j \in B^c}} \int_{\Gamma_{\xi}}\frac{|m_k(\xi_1,\dots,\xi_k)| \langle \xi \rangle ^{s_0+\epsilon_0}}{\prod_{j=1}^{k_0} \langle \xi_j \rangle ^{s_j}} \mathbbm{1}_{|\Phi-\alpha| < M} d\xi_{j \in B} \\\lesssim &
	\sup_{\alpha, \xi_{j \in B^c}} \int_{\Gamma_{\xi}}\frac{|m_{k_0}(\xi_1,\dots,\xi_{k_0})| \langle \xi \rangle ^{s_0+\epsilon_0}}{\prod_{j=1}^{k_0} \langle \xi_j \rangle ^{s_j}} \mathbbm{1}_{|\Phi_{k_0}-\sum_{j=k_0+1}^kL(\xi_j)-\alpha| < M} d\xi_{j \in B}  \\\lesssim & 	\sup_{\sigma, \tilde{\alpha}, \xi_{j \in A^c}} \int_{\Gamma_{\xi}^\sigma}\frac{|m_{k_0}(\xi_1,\dots,\xi_{k_0})| \langle \xi \rangle ^{s_0+\epsilon_0}}{\prod_{j=1}^{k_0} \langle \xi_j \rangle ^{s_j}} \mathbbm{1}_{|\Phi_{k_0}-\tilde{\alpha}| < M} d\xi_{j \in A}  \lesssim M^{1^-}.
\end{align*}
We are left with the frequency-restricted estimate for $\mathcal{M}_2$. If $s\le d/2$,
\begin{align*}
\mathcal{M}_2=\frac{|m_{k}(\xi_1,\dots,\xi_{k})|\langle \xi \rangle ^{s'_0+2\Delta s+\epsilon'_0}}{\prod_{j=1}^{k_0} \jap{\xi_j}^{s'_j+2\Delta s} \cdot \prod_{j=k_0+1}^k \jap{\xi_j}^{2s}} &\lesssim \frac{|m_{k_0}(\xi_1,\dots,\xi_{k_0})|\langle \xi \rangle ^{s'_0 + \varepsilon'_0}}{\langle \xi_1 \rangle ^{s'_1} \prod_{j=2}^{k_0} \langle \xi_j \rangle ^{s_j+\frac{(k-k_0)(d^+-2s)}{k_0-1}} \prod_{j=k_0+1}^{k} \langle \xi_{j} \rangle^{2s}} \\&\lesssim \frac{|m_{k_0}(\xi_1,\dots,\xi_{k_0})| \langle \xi \rangle ^{s'_0 + \varepsilon'_0}}{\prod_{j=1}^{k_0} \langle \xi_j \rangle ^{s'_j}} \prod_{j=k_0+1}^{k} \frac{1}{\langle \xi_{j} \rangle^{d^+}},
\end{align*}
where we moved the extra weights from the larger frequencies to the lower ones. If $s> d/2$, we have directly
\begin{align*}
	\mathcal{M}_2=\frac{|m_{k}(\xi_1,\dots,\xi_{k})| \langle \xi \rangle ^{s'_0+2\Delta s+\epsilon'_0}}{\prod_{j=1}^{k_0} \jap{\xi_j}^{s'_j+2\Delta s} \cdot \prod_{j=k_0+1}^k \jap{\xi_j}^{2s}} &\lesssim \frac{|m_{k_0}(\xi_1,\dots,\xi_{k_0})| \langle \xi \rangle ^{s'_0 + \varepsilon'_0}}{\prod_{j=1}^{k_0} \langle \xi_j \rangle ^{s'_j}} \prod_{j=k_0+1}^{k} \frac{1}{\langle \xi_{j} \rangle^{d^+}}.
\end{align*}
In any case, we can estimate
\begin{align*}
&\sup_{\alpha, \xi_{j \in B}} \int_{\Gamma_{\xi}}\mathcal{M}_2 \mathbbm{1}_{|\Phi-\alpha| < M} d\xi_{j \in B^c}\\\lesssim 	&\sup_{\alpha, \xi_{j \in B}} \int_{\Gamma_{\xi}}\frac{|m_{k_0}(\xi_1,\dots,\xi_{k_0})| \langle \xi \rangle ^{s'_0 + \varepsilon'_0}}{\prod_{j=1}^{k_0} \langle \xi_j \rangle ^{s'_j}}\prod_{j=k_0+1}^{k} \frac{1}{\langle \xi_{j} \rangle^{d^+}} \mathbbm{1}_{|\Phi-\alpha| < M} d\xi_{j \in B^c} \\\lesssim &
\sup_{\alpha, \xi_{j \in B}} \int \left(\int_{\Gamma_{\xi}^\sigma}\frac{|m_{k_0}| \langle \xi \rangle ^{s'_0 + \varepsilon'_0}}{\prod_{j=1}^{k_0} \langle \xi_j \rangle ^{s'_j}} \ \mathbbm{1}_{|\Phi_{k_0}-\sum_{j=k_0+1}^kL(\xi_j)-\alpha| < M} d\xi_{j \in A^c}\right) \prod_{j=k_0+1}^{k} \frac{1}{\langle \xi_{j} \rangle^{d^+}} d\xi_{k_0+1}\dots d\xi_k
\\\lesssim &
\left(\sup_{\sigma, \tilde{\alpha}, \xi_{j \in B}}  \int_{\Gamma_{\xi}^\sigma}\frac{|m_{k_0}| \langle \xi \rangle ^{s'_0 + \varepsilon'_0}}{\prod_{j=1}^{k_0} \langle \xi_j \rangle ^{s'_j}} \ \mathbbm{1}_{|\Phi_{k_0}-\tilde{\alpha}| < M} d\xi_{j \in A^c}\right) \prod_{j=k_0+1}^{k} \int  \frac{1}{\langle \xi_{j} \rangle^{d^+}} d\xi_{k_0+1}\dots d\xi_k\\\lesssim &\ M^{1^-}.
\end{align*}
Therefore \eqref{eq:fre} holds and the result follows by Lemma \ref{lem:interpol}.
\end{proof}

\begin{nb}\label{rem:conjugate}
	In the presence of complex conjugates, the proof follows exactly the same steps. The only observation is that there is no way of ensuring whether the $k_0$ largest frequencies correspond to $u$ or $\bar{u}$. As such, we must consider all possible combinations of signs in $\Phi_{k_0}$. This is why we suppose that the flexible frequency-restricted estimate must hold for all possible $k_0$-descents of $\Phi$. 
\end{nb}

\begin{nb}[Avoidable values of $\sigma$]\label{rem:avoid}
Let $c\in \R$ such that
\begin{equation}\label{eq:avoid}
	c\neq 1-\frac{k_0}{K},\quad k_0\le K\le k.
\end{equation}
Then we claim that  Theorem \ref{thm:ind_k} still holds if, in the definition of flexible frequency-restricted estimate, we may replace $\Gamma_\xi^\sigma$ with
$$
\Gamma_\xi^\sigma(c)=\Gamma_\xi^\sigma\cap \{(\xi_1,\dots,\xi_{k_0})\in (\R^d)^{k_0}: |\xi_1|\sim \dots \sim |\xi_{k_0}| \Rightarrow \sigma \not\simeq c\xi\}.
$$
To that end, it suffices to check that, if $|\xi_1|\sim \dots \sim |\xi_{k_0}|$ in the above proof, we can ensure that
$$
\sigma:= \sum_{j=k_0+1}^k \xi_j\quad \mbox{satisfies}\quad \sigma\not\simeq c\xi.
$$
Suppose that we are in a region where
$$|\xi_1|\sim \dots \sim |\xi_{k_0}|\quad \mbox{and}\quad \sigma\simeq c\xi
$$
Let $K$ be the number of frequencies which are comparable with $|\xi_1|$ (which must be between $k_0$ and $k$). If $K=k_0$, then $\sigma\simeq 0$ contradictng $\sigma\simeq c\xi$. If there are two frequencies, $\xi_j$ and $\xi_{j'}$, $j\le k_0<j'\le K$ with $\xi_j\not\simeq \xi_{j'}$, the permutation $j\leftrightarrow j'$ would then guarantee that $\sigma\not\simeq c\xi$ (and since these two frequencies are comparable, \eqref{eq:order} still holds).

As such, the only problematic case is when $\xi_j\simeq \xi_1$ for all $j\le K$, which we claim that is incompatible with $\sigma\simeq c\xi$. Indeed, this implies that
$$
\xi\simeq \xi_1+\dots + \xi_K \simeq K\xi_1
$$ 
and thus all large frequencies must be close to $\xi/K$. By the definition of $\sigma$,
$$
\sigma = \sum_{j=k_0+1}^k \xi_j = \xi-\sum_{j=1}^{k_0} \xi_j \simeq \left(1-\frac{k_0}{K}\right)\xi,
$$
contradicting $\sigma\simeq c\xi$.

In conclusion, when one derives the flexible frequency-restricted estimate, if all frequencies are comparable, we can \textit{avoid} the case $\sigma\simeq c\xi$ for any given $c$ satisfying \eqref{eq:avoid}.
\end{nb}

\medskip

We now move to the proof of the induction in terms of the dimension $d$.

\begin{proof}[Proof of Theorem \ref{thm:ind_d}]
To simplify the exposition, we present the proof when there are no complex conjugates in the nonlinearity. We consider the duality version of \eqref{eq:multi},
\begin{equation}
	I:=\left|\int_{\tau=\sum_j \tau_j} \int_{\Gamma_\xi} \frac{m(\xi)\jap{\xi}^{s+\epsilon}\jap{\tau-L(\xi)}^{b'}}{\prod_{j=1}^k\jap{\xi_j}^s\jap{\tau_j-L(\xi_j)}^b}v(\tau,\xi)\prod_{j=1}^k v_j(\tau_j,\xi_j) d\xi_1\dots d\xi_k d\tau_1\dots d\tau_k\right| \lesssim \|v\|_{L^2}\prod_{j=1}^k\|v_j\|_{L^2}.
\end{equation}
Without loss in generality, we can assume that $|\xi_1|\gtrsim |\xi|$ and that $\xi=(\eta,\zeta)\in \R^{d_0}\times \R^{d-d_0}$, with $|\eta|\sim |\xi|$. Define $\tilde{s}=s-\Delta s$, for $\Delta s=\left(\frac{d-d_0}{2}\right)^+$, and notice that, since $s>s_c(k,d)$ and $\epsilon<\epsilon_c(k,d,s)$,
$$
\tilde{s}>s_c(k,d_0),\quad \epsilon<\epsilon_c(k,d_0,\tilde{s}).
$$
Then, since $\tilde{s}>0$,
$$
\frac{|m(\xi)|\jap{\xi}^{s+\epsilon}}{\prod_{j=1}^k\jap{\xi_j}^s}\lesssim \frac{|m(\xi)|\jap{\xi}^{\tilde{s}+\epsilon}}{\jap{\xi_1}^{\tilde{s}}\prod_{j=2}^k\jap{\xi_j}^{\tilde{s}+\Delta s}}\lesssim \frac{|m_0(\eta)|\jap{\eta}^{\tilde{s}+\epsilon}}{\prod_{j=1}^k\jap{\eta_j}^{\tilde{s}}}\cdot \prod_{j=2}^k\frac{1}{\jap{\zeta_j}^{\Delta s}}.
$$
As such, 
\begin{align*}
	I&\lesssim \int_{\tau=\sum_j \tau_j} \int_{\Gamma_\xi} \frac{|m_0(\eta)|\jap{\eta}^{\tilde{s}+\epsilon}\jap{\tau-L(\xi)}^{b'}}{\prod_{j=1}^k\jap{\eta_j}^{\tilde{s}}\jap{\tau_j-L(\xi_j)}^b}\left(\prod_{j=2}^k\frac{1}{\jap{\zeta_j}^{\Delta s}}\right)v\prod_{j=1}^k v_j d\xi_1\dots d\xi_k d\tau_1\dots d\tau_k \\&\lesssim \int_{\Gamma_\zeta}\left(\int_{\tau=\sum_j \tau_j} \int_{\Gamma_\eta} \frac{|m_0(\eta)|\jap{\eta}^{\tilde{s}+\epsilon}\jap{\tau-L(\xi)}^{b'}}{\prod_{j=1}^k\jap{\eta_j}^{\tilde{s}}\jap{\tau_j-L(\xi_j)}^b}v\prod_{j=1}^k v_j d\eta_1\dots d\eta_k d\tau_1\dots d\tau_k\right) \prod_{j=2}^k\frac{1}{\jap{\zeta_j}^{\Delta s}} d\zeta_1\dots d\zeta_k\\&=: \int_{\Gamma_\zeta} I(\zeta,\zeta_1,\dots, \zeta_k) \prod_{j=2}^k\frac{1}{\jap{\zeta_j}^{\Delta s}} d\zeta_1\dots d\zeta_k.
\end{align*}
Proceeding exactly as in the proof of \cite[Lemma 2]{cos} and using the frequency-restricted estimate \eqref{eq:fre_d},
$$
|I(\zeta, \zeta_1,\dots, \zeta_k)|\lesssim \|v(\zeta)\|_{L^2_{\tau,\eta}}\prod_{j=1}^k \|v_j(\zeta_j)\|_{L^2_{\tau,\eta}}.
$$
uniformly in $\zeta_1,\dots, \zeta_k$.
Therefore, by Cauchy-Schwarz,
\begin{align*}
	I&\lesssim \int_{\Gamma_\zeta} \|v(\zeta)\|_{L^2_{\tau,\eta}}\prod_{j=1}^k \|v_j(\zeta_j)\|_{L^2_{\tau,\eta}} \prod_{j=2}^k\frac{1}{\jap{\zeta_j}^{\Delta s}} d\zeta_1\dots d\zeta_k\\& \lesssim \sup_\zeta \left(\int_{\Gamma_\zeta}\prod_{j=2}^k\frac{1}{\jap{\zeta_j}^{2\Delta s}} d\zeta_2\dots d\zeta_k \right)^{\frac{1}{2}} \|v\|_{L^2}\prod_{j=1}^k \|v_j\|_{L^2} \lesssim  \|v\|_{L^2}\prod_{j=1}^k \|v_j\|_{L^2} 
\end{align*}
and the proof is finished.
\end{proof}

\begin{nb}
	The specific choice of sets $A$ and multipliers $\mathcal{M}_1,\mathcal{M}_2$ in a frequency-restricted estimate may depend on specific regions in the frequency space. This poses no difficulty, as one can first split the multilinear estimate into the various regions, use the frequency-restricted estimate to prove the multilinear estimate and them sum all contributions (as long as the number of such regions is finite).
\end{nb}

\section{Application I: the generalized Korteweg-de Vries equation}\label{sec:kdv}

In this section, we prove Theorem \ref{thm:gkdv}. In this case,
$$
\Phi=\xi^3-\sum_{j=1}^k \xi_j^3
$$
and the Fourier multiplier for the nonlinear term is
$$
m(\xi_1,\dots, \xi_k)=\xi.
$$
Clearly, this multiplier satisfies \eqref{eq:ind_multiplier}. As such, Theorem \ref{thm:gkdv} is a direct consequence of Theorem \ref{thm:ind_k} once we prove the flexible frequency-restricted estimate for $k=5$.

\begin{prop}\label{prop:5kdv}
	Fix $k=5$. Given $s>s_c(5,1)=0$ and $\epsilon <\epsilon_c(5,1,s)=\min\{4s,1\}$, the flexible frequency-restricted estimate holds.
\end{prop}
\begin{proof}
	We begin by ordering $|\xi_1|\ge \dots \ge |\xi_5|$, which implies $|\xi_1|\gtrsim |\xi|$. We can suppose that $|\xi|\ge |\xi_1|$: if not,
	$$
	\frac{|\xi|\jap{\xi}^s}{\jap{\xi_1}^s}\lesssim 	\frac{|\xi_1|\jap{\xi_1}^s}{\jap{\xi}^s}
	$$
	and we may exchange $\xi$ with $\xi_1$. If $|\xi|\lesssim 1$, the multiplier is bounded and so is the domain of integration, meaning that we can bound \eqref{eq:ffre1}-\eqref{eq:ffre2} directly. Hence we suppose $|\xi|\gg 1$ from now on.
	
\medskip
	
\noindent \textbf{Case 1.} $|\xi_1| \not \simeq |\xi_3|$, $|\xi| \not \simeq |\xi_2|$.	We take $A=\{1,3,5\}$ and we must show
	\begin{equation}
		\label{estimate1.a}
		\sup_{\sigma, \alpha, \xi, \xi_2, \xi_4} \int_{\Gamma_{\xi}^{\sigma}} \frac{|\xi| \langle \xi \rangle^{s+\varepsilon}}{\langle \xi_1 \rangle^{s} \langle \xi_3 \rangle^{2s} \langle \xi_5 \rangle^{2s}} \mathbbm{1}_{|\Phi-\alpha| < M}\, d\xi_1 \, d\xi_5 \lesssim M^{1^-}
	\end{equation}
	and
	\begin{equation}
		\label{estimate1.b}
		\sup_{\sigma, \alpha, \xi_1, \xi_3, \xi_5} \int_{\Gamma_{\xi}^{\sigma}} \frac{|\xi| \langle \xi \rangle^{s+\varepsilon}}{\langle \xi_1 \rangle^{s} \langle \xi_2 \rangle^{2s}\langle \xi_{4} \rangle^{2s}} \mathbbm{1}_{|\Phi-\alpha| < M}\, d\xi \, d\xi_4 \lesssim M^{1^-}.
	\end{equation}
	For the first estimate, let us fix $\xi, \xi_2, \xi_4$. Observe that
	\begin{equation}
	\left|\frac{\partial \Phi}{\partial \xi_1}\right| = 3 |\xi_1^2-\xi_3^2| \gtrsim |\xi_1|^2 \sim |\xi|^2
	\end{equation}
	and so we can perform the change of variable $\xi_1 \mapsto \Phi$. Let $\eta>0$ be sufficiently small. We have
	\begin{equation}
		\begin{split}
			&\  \int_{\Gamma_{\xi}^{\sigma}} \frac{|\xi| \langle \xi \rangle^{s+\varepsilon}}{\langle \xi_1 \rangle^{s}\langle \xi_3 \rangle^{2s} \langle \xi_5 \rangle^{2s}}  \mathbbm{1}_{|\Phi-\alpha| < M} \, d\xi_1\, d\xi_5
			\\\lesssim &\  \int_{\Gamma_{\xi}^{\sigma}} \frac{ |\xi|e^{1+\varepsilon}}{\langle \xi_5 \rangle^{4s}} \mathbbm{1}_{|\Phi-\alpha| < M} \, d\xi_1\, d\xi_5
			\\\leq &\ \Bigg(\int_{\Gamma_{\xi}^{\sigma}} \frac{| \xi|^{(1+\varepsilon)(1+\eta)}}{\langle \xi_5 \rangle^{4s(1+\eta)}} \mathbbm{1}_{|\Phi-\alpha| < M} \, d\xi_1\, d\xi_5\Bigg)^{\frac{1}{1+\eta}} \Bigg(\int_{|\xi_1|,|\xi_5| \leq |\xi|} 1 \, d\xi_1\, d\xi_5\Bigg)^{\frac{\eta}{1+\eta}}
			\\\lesssim &\  |\xi|^{1+\varepsilon+\frac{2 \eta}{1+\eta}} \Bigg(\int_{\Gamma_{\xi}^{\sigma}} \frac{1}{\langle \xi_5 \rangle^{4s(1+\eta)}} \mathbbm{1}_{|\Phi-\alpha| < M} \frac{1}{|\xi|^2} \, d\Phi\, d\xi_5\Bigg)^{\frac{1}{1+\eta}}
			\\\lesssim &\ |\xi|^{\varepsilon-1+\frac{4 \eta}{1+\eta}} M^{\frac{1}{1+\eta}} \Bigg(\int \frac{1}{\langle \xi_5 \rangle^{4s(1+\eta)}} \, d\xi_5\Bigg)^{\frac{1}{1+\eta}}\\ \lesssim&\  |\xi|^{\varepsilon-1+\frac{4 \eta}{1+\eta}} M^{\frac{1}{1+\eta}} \min\{ 1, |\xi|^{4s+\frac{1}{1+\eta}}\} \lesssim M^{\frac{1}{1+\eta}},
		\end{split}
	\end{equation}
as long as $\epsilon<\min\{4s,1\}$. Thus \eqref{estimate1.a} follows. By symmetry, the proof of \eqref{estimate1.b} follows the exact same steps. 

\begin{nb}
	Before we proceed to the remaining cases, let us simplify the exposition by deriving frequency-restricted estimates with a full power of $M$. To obtain a power slightly smaller than $1$, it suffices to apply Hölder's inequality as shown above. This remark also applies to Sections \ref{sec:zk} and \ref{sec:nls}.
\end{nb}

\medskip

\noindent\textbf{Case 2.} $|\xi_1| \not \simeq |\xi_3|$, $|\xi| \simeq |\xi_2|$. In this case, $|\xi_4| \not \simeq |\xi|$, and we claim further that $|\xi_4| \gtrsim |\xi_2|$. Indeed, if this did not happen, we would have $\xi \simeq \xi_1+\xi_2+\xi_3$, with $\xi_1 \simeq \pm \xi$ and $\xi_2 \simeq \pm \xi$. Therefore, we would need $\xi_3 \simeq \pm\xi$ or $\xi_3\simeq 3\xi$, which is impossible since $\xi_3 \not \simeq \xi_1 \simeq \xi$ and $|\xi_3| \leq |\xi_1|$.
We then prove that
\begin{equation}
	\label{estimate2.a}
	\sup_{\sigma, \alpha, \xi, \xi_2, \xi_4} \int_{\Gamma_{\xi}^{\sigma}} \frac{|\xi| \langle \xi \rangle^{s+\varepsilon}}{\langle \xi_1 \rangle^{s} \langle \xi_3 \rangle^{2s} \langle \xi_5 \rangle^{2s}} \mathbbm{1}_{|\Phi-\alpha| < M}\, d\xi_1 d\xi_5 \lesssim M
\end{equation}
and
\begin{equation}
	\label{estimate2.b}
	\sup_{\sigma, \alpha, \xi_1, \xi_3, \xi_5} \int_{\Gamma_{\xi}^{\sigma}} \frac{|\xi| \langle \xi \rangle^{s+\varepsilon}}{\langle \xi_1 \rangle^{s} \langle \xi_2 \rangle^{2s}\langle \xi_{4} \rangle^{2s}} \mathbbm{1}_{|\Phi-\alpha| < M}\, d\xi d\xi_2 \lesssim M
\end{equation}
The first estimate is completely analogous to the one in the first case. For the second one is similar: since
\begin{equation}
	\left|\frac{\partial \Phi}{\partial \xi}\right| = 3 |\xi^2-\xi_4^2| \geq 3 |\xi|^2,
\end{equation}
we perform the change of variables $\xi\mapsto \Phi$ and use $|\xi_4|\gtrsim |\xi_2|$ to obtain
\begin{align*}
	\int_{\Gamma_{\xi}^{\sigma}} \frac{|\xi| \langle \xi \rangle^{s+\varepsilon}}{\langle \xi_1 \rangle^{s} \langle \xi_2 \rangle^{2s}\langle \xi_{4} \rangle^{2s}} \mathbbm{1}_{|\Phi-\alpha| < M}\, d\xi d\xi_2 \lesssim \int_{\Gamma_{\xi}^{\sigma}} \frac{|\xi|^{1+\epsilon}}{\langle \xi_2 \rangle^{4s}} \mathbbm{1}_{|\Phi-\alpha| < M}\, \frac{1}{|\xi|^2}d\Phi d\xi_2 \lesssim M.
\end{align*}

\medskip

\noindent\textbf{Case 3.} $|\xi_1| \simeq |\xi_3|$, $|\xi_1| \not \simeq |\xi_5|$, $|\xi| \not \simeq |\xi_2|$ (which implies in particular that $|\xi| \not \simeq |\xi_4|$).

\medskip

\noindent\textbf{Subcase 3.1} $|\xi_5| \gtrsim |\xi_3|$. Here, we show that
\begin{equation}
	\label{estimate3.1.a}
	\sup_{\sigma, \alpha, \xi, \xi_2, \xi_4} \int_{\Gamma_{\xi}^{\sigma}} \frac{|\xi| \langle \xi \rangle^{s+\varepsilon}}{\langle \xi_1 \rangle^{s} \langle \xi_3 \rangle^{2s} \langle \xi_5 \rangle^{2s}} \mathbbm{1}_{|\Phi-\alpha| < M}\, d\xi_1 \, d\xi_3 \lesssim M^{1^-}
\end{equation}
and
\begin{equation}
	\label{estimate3.1.b}
	\sup_{\sigma, \alpha, \xi_1, \xi_3, \xi_5} \int_{\Gamma_{\xi}^{\sigma}} \frac{|\xi| \langle \xi \rangle^{s+\varepsilon}}{\langle \xi_1 \rangle^{s} \langle \xi_2 \rangle^{2s}\langle \xi_{4} \rangle^{2s}} \mathbbm{1}_{|\Phi-\alpha| < M}\, d\xi \, d\xi_4 \lesssim M^{1^-}.
\end{equation}
The second estimate is completely analogous to \eqref{estimate1.b}. For the first, the fact that $|\partial_{\xi_1}\Phi|\gtrsim |\xi|^2$ and $|\xi_5|\gtrsim |\xi_3|$ allows us to proceed as in Case 2.

\medskip

\noindent\textbf{Subcase 3.2.}
$|\xi_5| \ll |\xi_3|$. In this case, we prove that
\begin{equation}
	\label{estimate3.2.b}
	\sup_{\sigma, \alpha, \xi_1, \xi_2, \xi_3} \int_{\Gamma_{\xi}^{\sigma}} \frac{|\xi| \langle \xi \rangle^{s+\varepsilon}}{\langle \xi_1 \rangle^{s} \langle \xi_4 \rangle^{2s}\langle \xi_{5} \rangle^{2s}} \mathbbm{1}_{|\Phi-\alpha| < M}\, d\xi \, d\xi_5 \lesssim M^{1^-}
\end{equation}
and
\begin{equation}
	\label{estimate3.2.a}
	\sup_{\sigma, \alpha, \xi, \xi_4, \xi_5} \int_{\Gamma_{\xi}^{\sigma}} \frac{|\xi| \langle \xi \rangle^{s+\varepsilon}}{\langle \xi_1 \rangle^{s} \langle \xi_2 \rangle^{2s} \langle \xi_3 \rangle^{2s}} \mathbbm{1}_{|\Phi-\alpha| < M}\, d\xi_1 \, d\xi_2 \lesssim M^{1^-}.
\end{equation}
For the first estimate, since $|\partial_\xi \Phi|\gtrsim |\xi|^2$,
$$
\int_{\Gamma_{\xi}^{\sigma}} \frac{|\xi| \langle \xi \rangle^{s+\varepsilon}}{\langle \xi_1 \rangle^{s} \langle \xi_4 \rangle^{2s}\langle \xi_{5} \rangle^{2s}} \mathbbm{1}_{|\Phi-\alpha| < M}\, d\xi d\xi_5 \lesssim \int_{|\xi_5|\le |\xi|} \frac{|\xi| \langle \xi \rangle^{\varepsilon}}{\langle \xi_{5} \rangle^{4s}} \mathbbm{1}_{|\Phi-\alpha| < M}\, \frac{1}{|\xi|^2}d\Phi d\xi_5 \lesssim M.
$$
For the second estimate, observe that $|\nabla_{\xi_1,\xi_2}\Phi|\ll |\xi|^2$.
In this stationary case, we replace $\Phi$ with a quadratic form using Morse's lemma. Define
\begin{equation}
	p_j := \frac{\xi_j}{\xi}\quad\mbox{ and } \quad P:=1-\sum_{j=1}^5 p_j^3.
\end{equation}
In particular, $\Phi=\xi^3P$ and we are near one of the stationary points\footnote{As $\xi_4, \xi_5$ and $\sigma$ are negligible, the resonance function behaves as the one of the modified KdV equation.}
$$
(p^0_1,p^0_2)= \left(\frac{1}{3},\frac{1}{3}\right),\quad (1,1),\quad (1,-1)\quad\mbox{or}\quad(-1,1).
$$
A direct computation shows that these critical points are non-degenerate. As such, by Morse's lemma (with parameters) \cite[Lemma C.6.1]{hormanderIII}, around each critical point, there exists a local change of coordinates $(p_1,p_2)\mapsto (q_1,q_2)$ such that
$$
P=P(p_1^0,p_2^0) + \lambda_1q_1^2 + \lambda_2q_2^2,\quad \lambda_1,\lambda_2\in\{\pm 1\},\quad |q_1|, |q_2|\ll1.
$$
Moreover, since the space of the parameters $\{p_4,p_5,\sigma/\xi\}$ is compact, the size of the local neighborhoods and the jacobian of the coordinate chart can be made uniform in $p_4, p_5$ and $\sigma$. As such,
\begin{align*}
	&\int_{\Gamma_{\xi}^{\sigma}} \frac{|\xi| \langle \xi \rangle^{s+\varepsilon}}{\langle \xi_1 \rangle^{s} \langle \xi_2 \rangle^{2s} \langle \xi_3 \rangle^{2s}} \mathbbm{1}_{|\Phi-\alpha| < M}\, d\xi_1  d\xi_2 \\ \lesssim & \int_{\Gamma_{\xi}^{\sigma}} |\xi|^{1+\epsilon-4s} \mathbbm{1}_{|\xi^3P-\alpha| < M}\, \xi^2dp_1 dp_2 \\ \lesssim & \int_{|q_1|, |q_2|\ll 1} |\xi|^{1+\epsilon-4s} \mathbbm{1}_{\left|\lambda_1q_1^2 + \lambda_2q_2^2-\frac{\alpha}{\xi^3}\right| < \frac{M}{|\xi|^3}}\, \xi^2dq_1 dq_2 \lesssim M^{1^-}.
\end{align*}

\medskip
\noindent \textbf{Case 4.} $|\xi|\simeq |\xi_1|\simeq  |\xi_2|\simeq |\xi_3|\not\simeq |\xi_5|$.

\medskip
\noindent \textbf{Subcase 4.1.} $|\xi_5| \gtrsim |\xi_3|$, $|\xi_4| \not \simeq |\xi|$ (which implies, in particular, that $|\xi_4|\gtrsim |\xi_2|$).  We take $A=\{1,3,5\}$ and we prove that

\begin{equation}
	\label{estimate4.1.a}
	\sup_{\sigma, \alpha, \xi, \xi_2, \xi_4} \int_{\Gamma_{\xi}^{\sigma}} \frac{|\xi| \langle \xi \rangle^{s+\varepsilon}}{\langle \xi_1 \rangle^{s} \langle \xi_3 \rangle^{2s} \langle \xi_5 \rangle^{2s}} \mathbbm{1}_{|\Phi-\alpha| < M}\, d\xi_1 \, d\xi_3 \lesssim M^{1^-}
\end{equation}
and
\begin{equation}
	\label{estimate4.1.b}
	\sup_{\sigma, \alpha, \xi_1, \xi_3, \xi_5} \int_{\Gamma_{\xi}^{\sigma}} \frac{|\xi| \langle \xi \rangle^{s+\varepsilon}}{\langle \xi_1 \rangle^{s} \langle \xi_2 \rangle^{2s}\langle \xi_{4} \rangle^{2s}} \mathbbm{1}_{|\Phi-\alpha| < M}\, d\xi \, d\xi_2 \lesssim M^{1^-}.
\end{equation}
This follows the same arguments as in Case 2, transferring the weight from the smallest frequency to the intermediate one and then performing a change of variables from the largest frequency to $\Phi$: for example,
\begin{equation}\label{eq:est41}
	 \int_{\Gamma_{\xi}^{\sigma}} \frac{|\xi| \langle \xi \rangle^{s+\varepsilon}}{\langle \xi_1 \rangle^{s} \langle \xi_3 \rangle^{2s} \langle \xi_5 \rangle^{2s}} \mathbbm{1}_{|\Phi-\alpha| < M}\, d\xi_1 \, d\xi_3 \lesssim  \int_{|\xi_3|\lesssim |\xi|} \frac{|\xi|^{1+\epsilon}}{ \langle \xi_3 \rangle^{4s} } \mathbbm{1}_{|\Phi-\alpha| < M}\, \frac{1}{\xi^2}d\Phi \, d\xi_3 \lesssim M.
\end{equation}
\medskip

\noindent\textbf{Subcase 4.2.} $|\xi_5| \gtrsim |\xi_3|$, $|\xi_4| \simeq |\xi|$. Observe that $|\xi_5|\not\simeq |\xi_1|\simeq |\xi_4|$. In this case we prove
\begin{equation}
	\label{estimate4.2.a}
	\sup_{\sigma, \alpha, \xi, \xi_4, \xi_5} \int_{\Gamma_{\xi}^{\sigma}} \frac{|\xi| \langle \xi \rangle^{s+\varepsilon}}{\langle \xi_1 \rangle^{s} \langle \xi_2 \rangle^{2s} \langle \xi_3 \rangle^{2s}} \mathbbm{1}_{|\Phi-\alpha| < M}\, d\xi_1 \, d\xi_2 \lesssim M^{1^-}
\end{equation}
and
\begin{equation}
	\label{estimate4.2.b}
	\sup_{\sigma, \alpha, \xi_1, \xi_2, \xi_3} \int_{\Gamma_{\xi}^{\sigma}} \frac{|\xi| \langle \xi \rangle^{s+\varepsilon}}{\langle \xi_1 \rangle^{s} \langle \xi_4 \rangle^{2s}\langle \xi_{5} \rangle^{2s}} \mathbbm{1}_{|\Phi-\alpha| < M}\, d\xi \, d\xi_4 \lesssim M^{1^-}
\end{equation}
The first estimate follows exactly as in Subcase 3.2. For the second estimate, one proceeds as in \eqref{eq:est41}.

\medskip

\noindent \textbf{Subcase 4.3.} $|\xi_5| \ll |\xi_3|$. We claim that $|\xi_4| \ll |\xi|$. Indeed, since $|\sigma|\lesssim |\xi_5|\ll |\xi|$,
$$
\xi_4 \simeq \xi-\sum_{j=1}^3 \xi_j \simeq k\xi,\quad \mbox{for some }k\in\{0,\pm2, 4\}.
$$
The cases $k=\pm 2, 4$ are impossible as $|\xi_4|\le |\xi_1|\simeq |\xi|$.
As such, the proof of \eqref{estimate4.2.a} follows as before, while \eqref{estimate4.2.b} is a consequence of the arguments used in Case 1:
$$
\int_{\Gamma_{\xi}^{\sigma}} \frac{|\xi| \langle \xi \rangle^{s+\varepsilon}}{\langle \xi_1 \rangle^{s} \langle \xi_4 \rangle^{2s}\langle \xi_{5} \rangle^{2s}} \mathbbm{1}_{|\Phi-\alpha| < M}\, d\xi \, d\xi_4 \lesssim \int_{|\xi_5|\ll |\xi|\simeq |\xi_1|} \frac{|\xi|^{1+\epsilon} }{ \langle \xi_{5} \rangle^{4s}} \mathbbm{1}_{|\Phi-\alpha| < M}\frac{1}{\xi^2} d\Phi d\xi_5 \lesssim M.
$$

\medskip 
\noindent\textbf{Case 5.} $|\xi_1| \simeq |\xi_3|$, $|\xi_1| \simeq |\xi_5|$. This is the fully stationary case, where all frequencies are comparable and the resonance is stationary regardless of the combination of frequencies. In this case, \eqref{estimate1.a} and \eqref{estimate1.b} follow from an application of Morse's lemma as in Subcase 3.2. 
\end{proof}

\begin{nb}
	In the above proof, one sees that the introduction of the flexibility parameter $\sigma$ causes almost no difficulty in deriving the frequency-restricted estimate. The fact that $\sigma\lesssim \min_j |\xi_j|$ was only needed briefly in Case 2 and Subcase 4.3 to exclude the possibility of a $high\times low$ interaction. 
\end{nb}

\begin{nb}
	If one aims only to prove local well-posedness (that is, with $\epsilon=0$), the proof can be greatly shortened, as the jacobian of the change of variables can be used to integrate in the remaining variable and there is no need to carefully transfer weights from one frequency to another.
\end{nb}

\section{Application II: the generalized Zakharov-Kuznetsov equation}\label{sec:zk}

\begin{proof}[Proof of Theorem \ref{thm:2dzk}]
By Proposition \ref{prop:5kdv}, the flexible frequency-restricted estimate holds for the quintic KdV. By Theorem \ref{thm:ind_k}, this implies that the frequency-restricted estimates for \eqref{gkdv} hold for any $k\ge 5$. Now observe that \eqref{zksym} has a linear operator of the form \eqref{eq:Ld}, with $L_0(\xi)=\xi^3$, and the nonlinear multiplier satisfies
$$
|m(\xi)|=|\xi^1 + \xi^2|\lesssim \max\{|\xi^1|, |\xi^2|\},\quad \xi=(\xi^1,\xi^2).
$$
As such, Theorem \ref{thm:2dzk} follows directly from Theorem \ref{thm:ind_d}.
\end{proof}

For $d\ge 3$, there is no symmetrization that allows us to induce the result from the generalized KdV and we are forced to work directly in dimension $d$. By induction in $k$, Theorem \ref{thm:gzk} will be a direct consequence of the following proposition.

\begin{prop}
	Fix $d\ge 3$. Given $s>s_c(3,d)=\frac{d}{2}-1$ and $\epsilon<\epsilon_c(3,d,s)=\min\{2s-d+2,1\}$, the flexible frequency-restricted estimate holds. 
\end{prop}
\begin{proof}
Throughout the proof, write $\xi=(\xi^1,\dots,\xi^d)$. Without loss of generality, we suppose that $|\xi|\ge |\xi_1|\ge |\xi_2|\ge |\xi_3|$ and prove

\begin{equation}\label{eq:dzk1}
	\sup_{\sigma, \alpha, \xi, \xi_2} \int_{\Gamma_{\xi}^{\sigma}} \frac{|\xi| \langle \xi \rangle^{s+\varepsilon}}{\langle \xi_1 \rangle^{s} \langle \xi_3 \rangle^{2s}} \mathbbm{1}_{|\Phi-\alpha| < M}\, d\xi_1 \lesssim M^{1^-}
\end{equation}
and
\begin{equation}\label{eq:dzk2}
	\sup_{\sigma, \alpha, \xi_1, \xi_3} \int_{\Gamma_{\xi}^{\sigma}} \frac{|\xi| \langle \xi \rangle^{s+\varepsilon}}{\langle \xi_1 \rangle^{s} \langle \xi_2 \rangle^{2s}} \mathbbm{1}_{|\Phi-\alpha| < M}\, d\xi_2 \lesssim M^{1^-}
\end{equation}
where
$$
\Phi=L(\xi)-\sum_{j=1}^3 L(\xi_j),\quad L(\xi)=\xi^1\sum_{j=1}^d (\xi^j)^2.
$$
If $|\xi|<1$, both integrals are bounded directly. Henceforth we take $|\xi|>1$.
\smallskip

\textit{Step 1. Exclusion of problematic cases.} We claim that, up to permutations of $\xi_1,\xi_2$ and $\xi_3$, we can always assume that
\begin{equation}\label{eq:claim}
	|\xi|\gtrsim |\xi_1|\gtrsim |\xi_2|\gtrsim |\xi_3|, \quad \xi\not\simeq \xi_2,\quad  \xi_1\not\simeq -\xi_3.
\end{equation}
First consider the case $|\xi_3|\ll|\xi_1|$. Then obviously $\xi_3\not\simeq -\xi_1$. Since $\xi\simeq \xi_1+\xi_2$, $\xi\simeq \xi_2$ would imply $\xi_1\simeq 0$, which is impossible since $|\xi|\sim |\xi_1|$.

Now we consider the case $|\xi_3|\gtrsim |\xi_1|$. In particular, independently on permetations of $\xi_1,\xi_2$ and $\xi_3$, we always have $|\xi|\gtrsim |\xi_1|\gtrsim |\xi_2|\gtrsim |\xi_3|$. Suppose that, for any such permutation,
$$
\xi\not\simeq \xi_2 \quad \mbox{and}\quad \xi_1\not\simeq -\xi_3
$$
fails. Then one would have
$$
(\xi_1,\xi_2,\xi_3)\simeq (\lambda_1\xi, \lambda_2\xi, \lambda_3\xi),\quad (\lambda_1,\lambda_2,\lambda_3)\in\{ (+,+,+), (+,-,-), (-,+,-), (-,-,+) \},
$$
and then either $\sigma\simeq 2\xi$ or $\sigma\simeq -2\xi$. By Remark \ref{rem:avoid}, both cases can be disregarded and claim \eqref{eq:claim} holds.
\smallskip

\textit{Step 2. Proof of the frequency-restricted estimates.} Having excluded the problematic regions, the proof follows the same lines as \cite[Proposition 2]{cos} (which corresponds to $\sigma=0$). We sketch the proof of \eqref{eq:dzk1}, as the proof of \eqref{eq:dzk2} is analogous. For $\xi,\xi_2$ fixed, define
$$
p_j=\frac{\xi_j}{|\xi|},\quad \Phi=|\xi|^3P(p_j).
$$
It was shown in \cite[Proposition 2]{cos} that $|\nabla_{p_1}P|\ll 1$ implies $|p_1|\simeq |p_3|$. Morever, since, by \eqref{eq:claim}, $p_1\not\simeq -p_3$, we always have the semi-nondegeneracy condition $\text{rank}(D^2_{p_1}P)\ge 2$.

\medskip
\noindent\textbf{Case 1.} $|\xi_3|\gtrsim |\xi|$. We split between nonstationary and stationary (but nondegenerate) cases.

\medskip
\noindent\textbf{Subcase 1.1.} $|\nabla_{p_1}P|\gtrsim 1$. Suppose that $|\partial_{p_1^1}P|\gtrsim 1$. Then performing the change of variables $p_1^1\mapsto P$,
\begin{align*}
		&\sup_{\sigma, \alpha, \xi, \xi_2} \int_{\Gamma_{\xi}^{\sigma}} \frac{|\xi| \langle \xi \rangle^{s+\varepsilon}}{\langle \xi_1 \rangle^{s} \langle \xi_3 \rangle^{2s}} \mathbbm{1}_{|\Phi-\alpha| < M}\, d\xi_1 \\\lesssim& \sup_{\sigma, \alpha, \xi, \xi_2} \int_{\Gamma_{\xi}^{\sigma}} \langle \xi \rangle^{d+1+\varepsilon-2s} \mathbbm{1}_{|P-\tilde{\alpha}| < \frac{M}{|\xi|^3}}\, dp_1^1...\,dp_1^d \\\lesssim& \sup_{\sigma, \alpha, \xi, \xi_2} \langle \xi \rangle^{d+1+\varepsilon-2s} \int_{\Gamma_{\xi}^{\sigma}} \mathbbm{1}_{|P-\tilde{\alpha}| < \frac{M}{|\xi|^3}}\, dP dp_1^2...\,dp_1^d\\\lesssim&\ M\cdot \sup_{\sigma, \alpha, \xi, \xi_2} \langle \xi \rangle^{d+1+\varepsilon-2s-3} \int_{|p_1^2|,\dots, |p_1^d|\lesssim 1}   dp_1^2...\,dp_1^d \lesssim M.
\end{align*}

\noindent\textbf{Subcase 1.2.} $|\nabla_{p_1}P|\ll 1$. In particular, we are near a critical point $z\in \R^d$ with $\nabla_{p_1}P(z)=0$. Since $\text{rank}( D^2_{p_1}P)\ge 2$, by Morse's Splitting lemma (with parameters) \cite[Theorem 8.3]{mawhinwillem}, there exists a local change of coordinates $p_1\mapsto q_1$, $|q_1|\ll 1$ such that
$$
P(p_1)=P(z)+\lambda_1 (q_1^1)^2 +\lambda_2 (q_1^2)^2 + h(q_1^3,\dots, q_1^d),\quad \lambda_1,\lambda_2=\pm 1.
$$
Moreover, since the space of parameters $\{q, q_2, \sigma/|\xi|\}$ is compact, the size of the local neighborhood can be made uniform.
As such, 
\begin{align*}
	&\sup_{\sigma, \alpha, \xi, \xi_2} \int_{\Gamma_{\xi}^{\sigma}} \frac{|\xi| \langle \xi \rangle^{s+\varepsilon}}{\langle \xi_1 \rangle^{s} \langle \xi_3 \rangle^{2s}} \mathbbm{1}_{|\Phi-\alpha| < M}\, d\xi_1 \\\lesssim& \sup_{\sigma, \alpha, \xi, \xi_2} \int_{\Gamma_{\xi}^{\sigma}} \langle \xi \rangle^{d+1+\varepsilon-2s} \mathbbm{1}_{|P-\tilde{\alpha}| < \frac{M}{|\xi|^3}}\, dp_1^1...\,dp_1^d \\\lesssim& \sup_{\sigma, \alpha, \xi, \xi_2} \langle \xi \rangle^{d+1+\varepsilon-2s} \int_{\Gamma_{\xi}^{\sigma}} \mathbbm{1}_{|P(z)+\lambda_1 (q_1^1)^2 +\lambda_2 (q_1^2)^2 + h(q_1^3,\dots, q_1^d)-\tilde{\alpha}| < \frac{M}{|\xi|^3}}\, dq_1^1\dots  dq_1^d\\\lesssim&\ M^{1^-}\cdot \sup_{\sigma, \alpha, \xi, \xi_2} \langle \xi \rangle^{d^++1+\varepsilon-2s-3} \int_{|q_1^3|,\dots, |q_1^d|\lesssim 1}   dp_1^3...\,dp_1^d \lesssim M^{1^-}.
\end{align*}

\noindent\textbf{Case 2.} $|\xi_3|\ll |\xi|$. Then we always have $|\nabla_{p_1}P|\gtrsim 1$. Assuming that $|\partial_{p_1^1}P|\gtrsim 1$,
\begin{align*}
	\sup_{\sigma, \alpha, \xi, \xi_2}  \int \frac{|\xi|\jap{\xi}^{s+\epsilon}}{\jap{\xi_1}^s\jap{\xi_3}^{2s}}\fia d\xi_1 & \lesssim \sup_{\sigma, \alpha, \xi, \xi_2}  \int_{|p_1|\lesssim 1} \frac{|\xi|^{1+\epsilon}}{\jap{\xi_3}^{2s}}\mathbbm{1}_{||\xi|^3P-\alpha|<M} |\xi|^ddp_1\\&\lesssim \sup_{\sigma, \alpha, \xi, \xi_2}  \int_{|p_1|\lesssim 1} \frac{|\xi|^{2+\epsilon}}{\jap{\xi_3}^{2s}}\mathbbm{1}_{||\xi|^3P-\alpha|<M} |\xi|^{d-1}dPdp_1^2\dots dp_1^d\\&\lesssim\sup_{\sigma, \alpha, \xi, \xi_2}  \int_{|p_1|\lesssim 1} \frac{|\xi|^{2+\epsilon}}{\jap{(\xi_3^2,\dots,\xi_3^d)}^{2s}}\mathbbm{1}_{||\xi|^3P-\alpha|<M}dPd\xi_1^2\dots d\xi_1^d\\&\lesssim M\cdot \sup_{\sigma, \alpha, \xi, \xi_2}  \int_{|(\xi_3^2,\dots,\xi_3^d)|\ll |\xi|} \frac{1}{\jap{(\xi_3^2,\dots,\xi_3^d)}^{2s}|\xi|^{1-\epsilon}}d\xi_1^2\dots d\xi_1^d\\&\lesssim M.
\end{align*}
\end{proof}

\section{Application III: the nonlinear Schrödinger equation}\label{sec:nls}
In this last section, we prove Theorem \ref{thm:nls}. First, we remark that the equation is in the conditions of Theorem \ref{thm:ind_d}, which reduces Theorem \ref{thm:nls} to dimension $d=2$. Using the induction in $k$, all that is left is to prove the flexible frequency-restricted estimate for $k=3$ and $d=2$. In this case, the $3$-descents of the resonance function are of the form
\begin{equation}\label{eq:res_nls}
	\Phi=|\xi|^2 + \lambda_1|\xi_1|^2 + \lambda_2 |\xi_2|^2 + \lambda_3|\xi_3|^2,\quad \lambda_j\in \{\pm 1\},\quad j=1,2,3.
\end{equation}
\begin{prop}
	Fix $k=3$ and $d=2$. Given $s>s_c(3,2)=0$ and $\epsilon<\epsilon_c(3,2,s)=\min\{2s,1\}$, the flexible frequency-restricted estimate holds for all resonance functions of the form \eqref{eq:res_nls}.
\end{prop}
\begin{proof}
Throughout this proof, we write $\xi_j=(\eta_j,\zeta_j)$, $j=\emptyset,1,2, 3$. Before we move to the details of the proof, let us explain briefly the rationale. First, since $s\ge 0$, we may always assume that $\xi$ is the largest frequency. Supposing that the frequencies are ordered as $ |\xi_1|\gtrsim |\xi_2|\gtrsim |\xi_3|$, we show that
\begin{equation}\label{eq:interpnls1}
\sup_{\sigma, \alpha, \xi, \xi_2} \int_{\Gamma_\xi^\sigma} \frac{\jap{\xi}^{s+\epsilon}}{\jap{\xi_1}^s\jap{\xi_3}^{2s}}\fia d\xi_3 \lesssim M^{1^-}
\end{equation}
and
\begin{equation}\label{eq:interpnls2}
	\sup_{\sigma, \alpha, \xi_1, \xi_3} \int_{\Gamma_\xi^\sigma} \frac{\jap{\xi}^{s+\epsilon}}{\jap{\xi_1}^s\jap{\xi_2}^{2s}}\fia d\xi_2 \lesssim M^{1^-}.
\end{equation}
The proof of each estimate is split according to the following:
\begin{itemize}
	\item $|\xi|\lesssim 1$, in which case the integral is uniformly bounded.
	\item The smallest frequency is not comparable with the largest one, in which case we claim that the resonance is not stationary in both \eqref{eq:interpnls1} and \eqref{eq:interpnls2}. Indeed, if $|\xi_3|\ll |\xi|$ and $\xi\simeq \pm \xi_2$, we would have
	$$
	\xi\simeq \xi_1+\xi_2 \simeq k\xi, \quad\mbox{for some } k\in\{0,\pm2\},
	$$ 
	which is impossible. This allows a change of variables in one direction and the integration of the weight $\jap{\cdot}^{2s}$ in the other. For the sake of convenience, we will always assume that the change of variables can be performed in the first component $\eta$. As such, \eqref{eq:interpnls1} follows from 
	\begin{equation}\label{eq:subcase12}
		\int_{\Gamma_\xi^\sigma} \frac{\jap{\xi}^{s+\epsilon}}{\jap{\xi_1}^s\jap{\xi_3}^{2s}}\fia d\xi_3 \lesssim \int_{|\zeta_3|\lesssim |\xi|} \frac{\jap{\xi}^{\epsilon}}{\jap{\zeta_3}^{2s}}\fia \frac{1}{|\xi|}d\Phi d\zeta_3 \lesssim M,
	\end{equation}
	while \eqref{eq:interpnls2} is a consequence of
	\begin{equation}\label{eq:subcase11}
		\int_{\Gamma_\xi^\sigma} \frac{\jap{\xi}^{s+\epsilon}}{\jap{\xi_1}^s\jap{\xi_2}^{2s}}\fia d\xi_2 \lesssim \int_{|\zeta_2|\lesssim |\xi|} \frac{\jap{\xi}^{\epsilon}}{\jap{\zeta_2}^{2s}}\fia \frac{1}{|\xi|}d\Phi d\zeta_2 \lesssim M.
	\end{equation}
	\item All frequencies are comparable, in which case the frequency weights can be transferred freely. As such, it suffices to check that we are either in a nonstationary case (where we proceed as above) or in a nondegenerate stationary case. In the latter, for $\xi,\xi_2$ fixed, we renormalize
	$$
	p_j=\frac{\xi_j}{|\xi|},\quad \Phi=|\xi|^2P,\quad |p_1-p_1^0|\ll 1,\quad \nabla P(p_1^0)=0.
	$$
	Since $D_{p_1}^2P$ is nondegenerate, there exists a change of coordinates\footnote{Usually this is ensured by Morse's lemma. Here, however, since $P$ is already a quadratic polynomial, one can make the computation exact, see for example the proof of \cite[Proposition 3]{cos}.} $p_1\mapsto q_1$ such that $$P=P(p_1^0) + \frac{1}{2}\langle D^2_{p_1}P\cdot q_1,q_1\rangle.$$ 
	Hence,
	\begin{align*}
		\int_{\Gamma_\xi^\sigma} \frac{\jap{\xi}^{s+\epsilon}}{\jap{\xi_1}^s\jap{\xi_3}^{2s}}\fia d\xi_3 &\lesssim \int_{|p_1-1|\ll1} |\xi|^{\epsilon-2s}|\xi|^2\mathbbm{1}_{||\xi|^2P-\alpha|<M}dp  \\&\lesssim \int_{|q_1|\ll1} |\xi|^{\epsilon-2s+2}\mathbbm{1}_{\left||\xi|^2\left(P(p_1^0) + \frac{1}{2}\langle D^2_{p_1}P\cdot q_1,q_1\rangle\right)-\alpha\right|<M}dq \lesssim M.
	\end{align*}
\end{itemize}
All that is left is to prove that, up to possible rearrangements between $\xi_1,\xi_2,\xi_3$, there are no degenerate stationary points when all frequencies are comparable. We now split the proof according to the different combinations of signs in \eqref{eq:res_nls}.

\medskip
\noindent\textbf{Case 1. $(+,+,+)$.}  For $\xi,\xi_2$ fixed, $D^2_{\xi_1}\Phi= 4I$, which is nondegenerate. The same happens for $\xi_1,\xi_3$ fixed.

\medskip
\noindent\textbf{Case 2. $(+,+,-)$.} For $\xi_1,\xi_3$ fixed, $D^2_\xi \Phi=4I$ and there are no degenerate critical points. For $\xi,\xi_2$ fixed, we have $D^2_{\xi_1}\Phi=0$, and thus we must guarantee that there are no critical points. In other words, we must avoid the region $\xi_1\simeq -\xi_3$. If $\xi_2\not\simeq-\xi_3$, it suffices to exchange $\xi_1$ with $\xi_2$. If $\xi_1\simeq \xi_2\simeq -\xi_3$, we prove instead
\begin{equation}\label{eq:interpnls3}
	\sup_{\sigma, \alpha, \xi, \xi_3} \int_{\Gamma_\xi^\sigma} \frac{\jap{\xi}^{s+\epsilon}}{\jap{\xi_1}^s\jap{\xi_2}^{2s}}\fia d\xi_2 \lesssim M^{1^-}
\end{equation}
and
\begin{equation}\label{eq:interpnls4}
	\sup_{\sigma, \alpha, \xi_1, \xi_2} \int_{\Gamma_\xi^\sigma} \frac{\jap{\xi}^{s+\epsilon}}{\jap{\xi_1}^s\jap{\xi_3}^{2s}}\fia d\xi_3 \lesssim M^{1^-}.
\end{equation}
As before, there are no degenerate critical points for $\xi,\xi_3$ fixed. For $\xi_1,\xi_2$ fixed, $|\nabla_\xi \Phi|\ll |\xi|$ would imply $\xi\simeq \xi_3$ and then
$$
\xi + \sigma=\xi_1+\xi_2+\xi_3 \simeq -\xi-\xi+\xi \simeq -\xi,\quad \mbox{that is,}\quad \sigma\simeq -2\xi,
$$
which is an avoidable value of $\sigma$ (see Remark \ref{rem:avoid}). Therefore there are no stationary points and \eqref{eq:interpnls4} follows as in \eqref{eq:subcase11}.

\medskip
\noindent\textbf{Case 3. $(-,+,-)$.} Similar to Case 1, for $\xi,\xi_2$ fixed, $D^2_{\xi_1}\Phi= -4I$, which is nondegenerate. For $\xi_1,\xi_3$ fixed,  $D^2_{\xi}\Phi= 4I$, and thus there are no degenerate points.

\medskip
\noindent\textbf{Case 4. $(-,-,-)$.} We proceed as in Case 2. For $\xi,\xi_2$ fixed, $D^2_{\xi_1}\Phi=-4I$. For $\xi_1,\xi_3$ fixed, we always have $D^2_\xi \Phi=0$, which means that we must ensure that there are no critical points. If one had critical points, even counting permuations of $\xi_1,\xi_2$ and $\xi_3$, then
$$
\xi \simeq \xi_1 \simeq \xi_2 \simeq \xi_3,
$$
which would imply that $\sigma\simeq 2\xi$. Once again, this is an avoidable value of $\sigma$. Consequently, we can always guarantee the nonexistence of critical points.
\end{proof}

\bibliography{biblio}
\bibliographystyle{plain}

\end{document}